\documentclass[11pt]{article}

\usepackage{amsmath,amscd,amsthm,amsfonts,amssymb}
\usepackage{xcolor}
\usepackage{tikz-cd} 
\usepackage{tikz}
\usepackage{chngcntr}
\counterwithin{figure}{subsection}

\usepackage[modulo]{lineno}
\usepackage{hyperref}
\usepackage{tikz}
\usepackage[latin1]{inputenc}
\numberwithin{figure}{section}

\usetikzlibrary{positioning}
	\newtheorem{thm}{Theorem}
	\newtheorem{cor}{Corollary}
	\newtheorem{lem}{Lemma}
	\newtheorem{rem}{Remark}

	\newtheorem{prop}{Proposition}

	\newtheorem*{ques}{Question}

	\newcounter{constant}

    \newcommand{\gaussp}{\mathfrak{g}}

        \newcommand{\gauss}{\mathfrak{g}}
	 
	 \newcommand{\st}{\mathbb{S}^3}
	    \newcommand{\cdeux}{\mathbb{C}^2}
	   \newcommand{\sd}{\mathbb{S}^2}
	  \newcommand{\sdeux}{\mathbb{S}^2}
	   \newcommand{\rtrois}{\mathbb{R}^3}
	 \newcommand{\rquatre}{\mathbb{R}^4}
	 \newcommand{\quat}{\mathbb{H}}
	\newcommand{\quatc}{\mathbb{H}_{\mathbb{C}}}
	  \newcommand{\quatU}{\mathbb{H}_U}
	    \newcommand{\quatP}{\mathbb{H}_P}
	   \newcommand{\quatUP}{\mathbb{H}_{UP}}

	      \providecommand{\keywords}[1]
{
  \small	
  \textrm{\textit{Keywords---}} #1
}

\begin{document}
	\title{\bf \fontsize{12}{12} \selectfont   ON THE SIZE OF MINIMAL SURFACES IN   $\mathbb{R}^4$  }
	\author{ \fontsize{10}{10} \selectfont    ARI AIOLFI,  MARC SORET AND MARINA VILLE}
	\date{ June 10 2021}

	\maketitle
	\begin{abstract} 
		   The Gauss map $g$  of    a  surface  $\Sigma$ in $\mathbb{R}^4$ takes its values in the Grassmannian  of oriented 2-planes of $\mathbb{R}^4$: $G^+(2,4) $.   
		   We give geometric criteria of  stability for  minimal surfaces in $\mathbb{R}^4$ in terms of $g$. We show in particular that if  the spherical area  of the Gauss map $ |g(\Sigma) |    $  of a minimal surface is smaller  than  $ 2\pi$ 
	  then the surface  is stable  by deformations which   fix the  boundary of the surface.  This answers the question of \cite{BDC3} in $\mathbb{R}^4$.
 	\end{abstract}
	\keywords{  Gauss map, Grassmannian, minimal surface, stability}
\section{\normalfont Introduction}
\let\thefootnote\relax\footnotetext{ \hskip -.15 in  2020 { \it Mathematics Subject Classification} . Primary  53A10.\\
 Supported  in part by   CNPq/Brazil \&  FSMP/France}
A geometric criterion for the stability of a  minimal surface $\Sigma$ in the Euclidean space $\mathbb{R}^3$ is the  spherical area of its Gauss map image (without multiplicity)
$|g(\Sigma)|$. Thus, 
if the area is smaller than $2\pi$ then the surface is stable   (see \cite{BDC} \& \cite{BDC2}  ). \\
A similar  stability  criterion  was later generalized for   simply-connected minimal surfaces in $\mathbb{R}^n$  where the Gauss map of $\Sigma$ takes its values in the 
Grassmannian $G^+(2,n)$ (see also   \cite{HO}):
 \begin{thm}[ \cite{BDC3} ]\label{thm3}
  Let  $\Sigma$ be a   minimal surface in  $\mathbb{R}^n$.  If $\Sigma$ is  simply-connected and  $|g(\Sigma)| \leq \frac{4\pi}{3}$ then $\Sigma$ is stable.
 \end{thm} 
 In \cite{BDC3}   one asks  whether the  weaker condition $|g(\Sigma)| < 2\pi$ ensures stability.  We prove in particular that  this is true for minimal surfaces  in $\mathbb{R}^4$.

More precisely,   we suppose   that $\Sigma$ be  a minimal surface of $\mathbb{R}^4$.  The Gauss map of  $\Sigma$ 
 takes its  values in the Grassmannian of oriented 2-planes in $\mathbb{R}^4$ $g: \Sigma \longrightarrow G^+(2,4)$.  Note that $G^+(2,4)$  is isometric 
to the product of spheres $\mathbb{S}^2\left(\frac{1}{\sqrt{2}}  \right) \times  \mathbb{S}^2\left(\frac{1}{\sqrt{2}}  \right) $ and that  $g$   is   the product of two  $\mathbb{S}^2$-valued  maps  
 : the left Gauss map and right Gauss map  which is denoted by $g = ( \gaussp_L, \gaussp_R)$.  As $\Sigma$ is minimal, it is a critical point for the area functional for deformations with fixed boundary. If it is a local minimum then it is stable.\\
We begin by   studying   the relationship between the stability of $\Sigma$ and the area of each projection of the Gauss map image of $\Sigma$ . We first show  :
\begin{thm} \label{thm1}
Let $\Sigma$ be a   minimal surface in  $\rquatre$ . Let $\lambda$ (resp. $\mu$ ) the first Dirichlet eigenvalues of the   Gauss images   $\gaussp_L(\Sigma)$  
( resp. $\gaussp_R(\Sigma)$ ). If the harmonic mean of $\lambda$ and $\mu$ is larger than 2 then $\Sigma$ is stable.
\end{thm}
We deduce from this theorem the Proposition \ref{hypequi}  (p. \pageref{page:hypequi1})  where the  stability condition  is expressed in terms of the area of the left and right Gauss map images.  This is obtained by replacing the 
  Dirichlet eigenvalues  by lowerbounds   in terms of the area {\it via }   the standard isoperimetric inequality for spherical domains as in \cite{BDC2}.
    One    derives  the  next corollary which directly  implies Theorem \ref{thm3}   of  \cite{BDC3} in $\mathbb{R}^4$.
 \begin{cor}\label{cor1}
Let $\Sigma$ be a   minimal surface in  $\rquatre$ . Suppose  that the  spherical areas  of $\Sigma$ satisfy 
$|\gaussp_L(\Sigma) |+  |\gaussp_R(\Sigma) | \leq \frac{4\pi}{3} $   then $\Sigma$ is stable.
\end{cor}
Notice first that there is no condition on the topology of $\Sigma$ because   only    the classical isoperimetric inequality on 2-spheres 
 is used ( and not an isoperimetric inequality for surfaces  in  $G(2,4)$ as in\cite{BDC3}). Secondly,  one always has  $|\gaussp_L(\Sigma) |+  |\gaussp_R(\Sigma) |   \leq | g(\Sigma)|$ 
  ( cf  Lemma \ref{kaehler} and Corollary \ref{aires}), hence   the stability 
 condition of Corollary \ref{cor1} on
the sum of the projected area   is a weaker hypothesis than the spherical area upperbound  of Theorem  \ref{thm3} of \cite{BDC3} \\
  The  stability domain  obtained in Corollary \ref{cor1}  (resp.  in Proposition \ref{hypequi} ) is represented by domain 1 (resp.  Domain 3)   of  figure  \ref{fig} 
  \begin{figure}[h!]\label{fig}
\begin{center}
\includegraphics[scale=0.4]{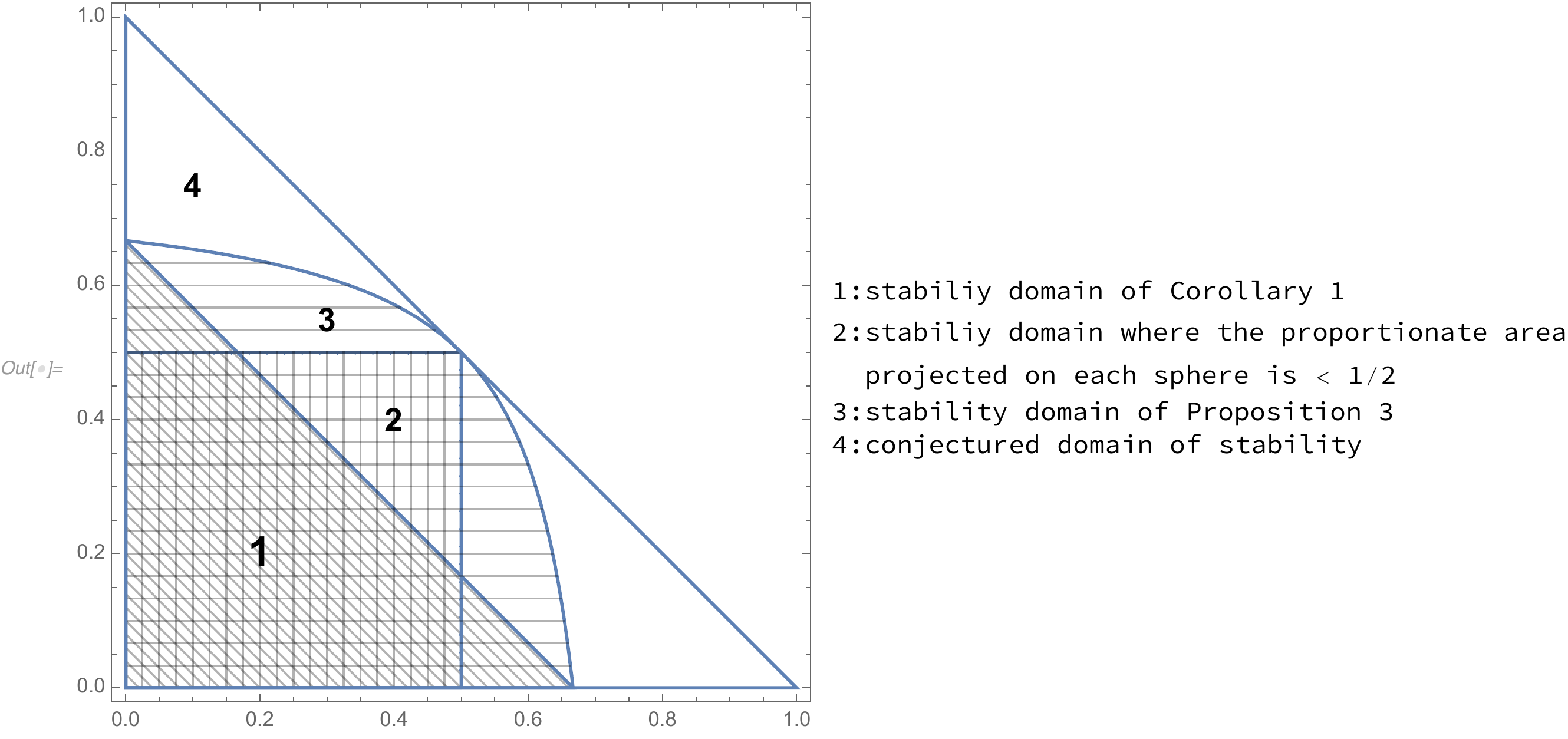} 
\caption{ Stability wrt to the projected spherical area :$ \Sigma \mapsto \left(\frac{ |\gaussp_L\left(\Sigma\right)|}{2\pi},\frac{ |\gaussp_R\left(\Sigma\right)|}{2\pi} \right)$  }
\end{center}
\end{figure} 
 
\begin{rem}\label{remarkone}
The spherical area  $|g(\Sigma)|$ or   $|\gaussp_L(\Sigma) |+  |\gaussp_R(\Sigma) | $ are  counted without multiplicity. 
In particular    $|\gaussp_L(\Sigma) |+  |\gaussp_R(\Sigma) | $   is  bounded from above by $4\pi$ contrary to $|g(\Sigma)|$ which has no {\it a priori}  upperbound ( see  Proposition  \ref{nobound}).
Notice also that  the  Corollary  1 can be restated  in terms of proportionate area by saying that if the proportionate area sum  is less than one third - then the minimal surface is stable  \end{rem}
Complex curves  $ \Sigma \subset \mathbb{R}^4$ satisfy  $|\gaussp_L(\Sigma)| =0$,   $|\gaussp_R(\Sigma)| \leq 2\pi$,   $|g(\Sigma)| \leq 2\pi$,   and are examples of stable 
minimal surfaces.    In  light of these examples, one  may wonder   -as   in  \cite{BDC3} -  if  the upperbound $\frac{4\pi}{3}$ of  Theorem \ref{thm3} can be replaced by  $2\pi$. Indeed, one shows  in the second part  of the paper that 
\begin{thm}\label{thm4}
 Let    $\Sigma$ be a   minimal surface in  $\mathbb{R}^4$.  If    $ |g(\Sigma)|   < 2\pi $ then   $\Sigma$  is stable.
\end{thm}
 The  proof relies on Theorem  \ref{thm1}  and on next proposition.
 \begin{prop}
 
For any  minimal surface $\Sigma$,  there is an associate minimal surface  $\Sigma^*$  isometric  to $\Sigma$,    such that: 
\begin{enumerate}
\item the left and right  Gauss area  of   $\Sigma^*$ are equal. 
\item  if   $\Sigma^*$ is stable than so is $\Sigma$.
\end{enumerate}
\end{prop}
Finally one  notices   in Section \ref{conjecture}   that the    domain  of stability  of Proposition \ref{hypequi}  and Corollary \ref{cor1} can be enlarged. 
Also, the consideration of complex curves show that  the domain of stability  as illustrated  in Figure \ref{fig}  must contain the  left and bottom edge of the square. 
Hence the  domain 4 of Figure   \ref{fig}  may  reasonnably represent a larger domain of stability    and   a stronger result than
Theorem  \ref{thm4} might  then  be true.
 \begin{ques}
 Let    $\Sigma$ be a   minimal surface in  $\mathbb{R}^4$.  If    $|\gaussp_L(\Sigma)| + |\gaussp_R(\Sigma)|  < 2\pi $ then  is  $\Sigma$  stable?
\end{ques}


\section{\normalfont Quaternions and Gauss maps }
The use of quaternions allows us to combine the classical approach to the Grassmannian $G^+(2,4)$  as a quadric in $\mathbb{C}P^3$ (see for example \cite{HO}) with Eells-Salamon's approach {\it via} complex structures on $\mathbb{R}^4$ (\cite{ES}).\\
We  identify  $\mathbb{R}^4$   with  $\cdeux$ and with  the quaternions $\quat =\mathbb{R}  \oplus \mathbb{R} I \oplus \mathbb{R}J \oplus \mathbb{R} K$ ; we choose the identification 
such that the canonical complex structure on  $\cdeux$  will be identified to $I$. Let us recall some basic facts.
 Each quaternion is the sum of a real and a pure quaternion $\quatP$: 
	$X = a + bI+cJ+dK = a + P_X$.
The product of two quaternions     is then given by: 
\begin{equation}\label{produit}
 X\cdot Y = aa' - <P_X,P_Y>  + a P_Y + a' P_X + P_X\wedge P_Y 
\end{equation}
where  $X =  a + P_X, Y   = a' + P_Y$\ and $\wedge$ is the cross product of $\rtrois$. The conjugate of  $X$ is denoted by $\overline{X} =  a - P_X$ and $\Re(X) = a$. 
The  Euclidean scalar product  on $\quat$ is then given by :
 $\langle X,Y\rangle :=  \Re(X\cdot\overline{Y}),   \quad  X, Y \in \quat $.
 We  identify  the  set of unit quats  : $\quatU$  with  the 3- sphere  $\st$ of radius $1$ around $O$  .
Similarly, we identify   the  set of pure quats  $\quatP$ - ie ${\rm span}( I,J,K)$-   with  $ \rtrois$.
With these identifications    $\quatUP := \quatP \cap\quatU $ is a  2-sphere $\sd$ which is an  equator of  the unit quats $\quatU= \st$.\\

 \subsection{ \normalfont The Grassmannian $G^+(2,4)$ } \label{grass}
 For any  oriented  vector plane $P \in  \quat$ choose any  orthogonal  positive basis of vectors of same length  $\{T_1,T_2\}$;  then $q := T_2.T_1^{-1}$ doesn't depend on the choice of 
 the  direct orthogonal basis of $P$  and $q$  is the only pure  unit quaternion
 -or complex  orthogonal structure-  that leaves the oriented plane invariant when operating on the left
 ($ q T_1 = T_2T_1^{-1}T_1 = T_2,$ and $qT_2 = q^2 T_1 = -T_1$. It is also   useful to notice that    $X^{-1} = \frac{\overline{X}}{|X|^2}$).\\
Proceeding  identically  for the  right multiplication
   we define a map from the Grassmannian of oriented planes of $\rquatre$ to $\sdeux \times \sdeux$:
 
   \begin{equation*}
  G = (g_L,  g_R) :
    \left(     
    \begin{array}{ll}
 G^+(2,4)  &\longrightarrow  \sd \times \sd \\
  P & \mapsto  (T_2 \cdot T_1^{-1},  T_1^{-1}\cdot  T_2)
  \end{array}
  \right)
   \end{equation*} 
As  $G$  is onto, $g$ is 1 to 1 and onto.\\
{\sl  We then identify  the Grassmannian  $ G^+(2,4)  $ with $\sd\times\sd$ via this    1-1 map  $G$}
where $\sd$ is the round sphere of radius one.
\begin{rem}
Notice that  $G^+(2,4)$ provided with its natural metric is isometric to $\sd \left( \frac{1}{\sqrt{2}}  \right) \times  \sd \left( \frac{1}{\sqrt{2}}  \right) $. Thus,  although 
we will make the computations in  $\sd\times\sd$,  results in  final statements  will be  expressed in terms of the Gauss maps with values in  the standard Grassmannian as in the introduction. 
\end{rem}
 
	\subsection{\normalfont  Left and right  Gauss maps of  $T\Sigma$ and $N\Sigma$ in isothermal coordinates}\label{gauss}
	Let $\Sigma$ be a $C^2$-surface. Around any point $p\in \Sigma$ we will  choose  local   isothermal   parameters (x,y) for the immersion
	$X: U \longrightarrow \Sigma\subset \quat$. Whether we consider the tangent bundle $T\Sigma$ or  the normal bundle $N\Sigma$ then two Gauss maps are generated.
	
Let us first consider the Gauss map of the tangent bundle.   At  each tangent plane  $T_p\Sigma$, we  compute,   at  point $p= X(x,y)$, the derivatives $X_x, X_y \in T_p\Sigma $ ; we  define  the  ( tangent)  Gauss map of  $\Sigma$  by :
 \begin{equation*}
 G := (\gaussp_L , \gaussp_R ): \left(
\begin{array}{cc} 
 \Sigma &  \longrightarrow   \sd \times \sd = \quat_{PU} \times   \quat_{PU}\\
 p & \mapsto   (X_y  \cdot    X_x^{-1} , X_x^{-1}  \cdot   X_y )
\end{array}
\right)
\end{equation*}
 
 $\gaussp_L (p) $  (resp.   $\gaussp_R (p) $ ) is the orthogonal  complex structure that leaves  the tangent  plane at $p$  invariant when acting  on the  left   
  (resp. on the right).\\
  We then consider the  Gauss map of the normal bundle. The useful information here is recalled in a lemma :
  \begin{lem}  Left  (resp. right) multiplication  by $\gaussp_L$   (resp. by  $\gaussp_R$ ) on $\quat$  both define rotations of $\quat$ such that the angle between any quaternion and its image 
  is of absolute value $\pi/2$.  Their respective actions are equal on $T\Sigma$ and of opposite sign on $N\Sigma$.
  \end{lem}
  \begin{proof}
  Suppose $q\in \quat_U$.  Then left or right multiplication  by $q$ defines an isometry:
  $\langle q\cdot x,  q\cdot x\rangle = \Re( qx\bar x\bar q) = |x|^2 $.
  And  $\langle q\cdot x,   x\rangle = \Re( qx\bar x) =\Re(q)  |x|^2 $.  Hence the isometry is a rotation  and the  absolute angle between any nonzero  quaternion and its image is constant.
 If $q\in \quatUP^*$ then then absolute angle is $\pi/2$.  By construction,   left multiplication  by $\gaussp_L$  and   right multiplication by  $\gaussp_R$ restricted to $T\Sigma$  are equal to the same $\pi/2$- rotation.
Since both  are isometries, the normal bundle  is  stable and the restriction of both rotations to $N\Sigma$  are   $\pi/2$-rotations.  If they were equal, 
then $\gaussp_Lx = x \gaussp_R$ for any $x\in\quat$. Hence  $\gaussp_L= \gaussp_R$
and $\gaussp_L $ commutes with any quaternion, which is impossible since $\gaussp_L $ is a nonzero pure quaternion. Hence the rotation of $\gaussp_L$  and $\gaussp_R$ on $N\Sigma$  are of opposite angle 
so  $\gaussp_L x =  x\gaussp_R^{-1} = -x\gaussp_R$ for any $x \in N\Sigma$ since the angle is $\pi/2$ .
  \end{proof}
  The action of $\gaussp_L$ on the left is a rotation by the lemma above. Hence it  preserves the orientation of $T_p\Sigma \oplus N_p\Sigma$  and, thus,  acts positively on $N_p\Sigma$.
Hence the  left  Gauss map $\gaussp_{N\Sigma, L} $  of the normal bundle    $N\Sigma$  verifies $\gaussp_{N\Sigma, L} = \gaussp_L$.
Moreover we deduce from the lemma  that the  right Gauss map  $\gaussp_{N\Sigma,R}$   of the  normal bundle  verifies $\gaussp_{N\Sigma,R} = \gaussp_R^{-1}$. 
 The expression of the normal Gauss map in  
 isothermal coordinates is thus given  by :
 \begin{equation*}
G_{N\Sigma} : \left(
\begin{array}{ll} 
 \Sigma &  \longrightarrow \sd \times \sd \\
 p & \mapsto   (  X_y\cdot X_x^{-1} ,  X_y^{-1} \cdot  X_x)

\end{array}
\right)
\end{equation*}
 \noindent and  $ G_{N\Sigma} =    ( Id,-Id)  \circ  G $. 
\begin{rem} We will not need it here but one can prove that $\mathfrak{g}_L$ (resp. $\mathfrak{g}_R$) is a rotation d'angle $+\frac{\pi}{2}$ (resp. $-\frac{\pi}{2}$) on $N\Sigma$.
\end{rem}
 
\section{\normalfont First derivatives of $\gaussp_L$  and $\gaussp_R$ for minimal surfaces}\label{trois}
In order to prove the results of stability we will need  some properties of the left and right Gauss maps.  We will use  local   isothermal coordinates  on a Riemann surface $U$ for the immersion 
$X : U \longrightarrow \Sigma \subset \quat$ with these notations :
$ E := |X_x|^2 = |X_y|^2, \langle X_x, X_y\rangle=0$ , the induced metric of $\Sigma$ on $U$ being $ds^2_\Sigma = E |dz|^2$ with $ z = x+iy$. 
\begin{prop} \label{propSurG}The left Gauss map $\gaussp_{L}$ and right Gauss map $\gaussp_{R}$  of the Gauss map  $ g:= (\gaussp_L,\gaussp_R) $  of the tangent bundle  of a minimal surface in $\mathbb{R}^4$ satisfy the following properties: 
\begin{enumerate}
\item  Left multiplication  by $\gaussp_{L}$ defines a complex structure on $T\Sigma$ and on $N\Sigma$. Right multiplication by $\gaussp_R$ defines the same complex structure on 
$T\Sigma$ but the antiholomorphic structure on $N\Sigma$. The values  of  $\gaussp_{L}$  and   $\gaussp_{R}$  are pure unitary quaternions and 
$$\gaussp_{L}^2 = \gaussp_{R}^2 = -1.$$
In particular,  $\gaussp_{L,x}$ and $\gaussp_L$  (resp.     $\gaussp_{R,x}$ and $\gaussp_R$) anticommute.
\item The Gauss maps $\gaussp_L$  and $\gaussp_R$ are anti- holomorphic  for the $\gaussp_L$- or   $\gaussp_R$ -complex structure on $\Sigma$. More precisely
$$ \gaussp_{L,y }= -\gaussp_L \gaussp_{L,x}, \quad  \gaussp_{R,y} =   - \gaussp_{R,x}    \gaussp_{R}. $$
Notice that the values of the maps  $\gauss_{X,a}$ are pure quaternions  where $ X = L,R, \quad a = x,y $.
\item    Let $B_{ab} $ be  the second fundamental form defined by the normal projections  of the second  derivatives of the position  vector $X_{ab}^N$, where $a,b = x \ {\rm or } \ y$. We have:
$$\gaussp_{L,x} =  -\gaussp_L ( B_{xx} + \gaussp_L B_{xy})X_x^{-1} ,\quad \gaussp_{R,x} =  - X_x^{-1} (B_{xx} +B_{xy}\gaussp_{R})\gaussp_{R}.$$
Notice that left multiplication by $\gaussp_{L,x}$  or   $\gaussp_{L,y}$  ( resp. right multiplication by $\gaussp_{R,x}$ or $\gaussp_{R,y}$) permutes $T\Sigma$ and $N\Sigma$ and are also elements
of $T\mathbb{S}^2$.
\item The quaternion fields
$$ B_{xx}+\gaussp_L B_{xy}  (= \gaussp_{L,y} X_x ) , \quad    B_{xx}+ B_{xy}\gaussp_R  (=X_x \gaussp_{R,y} ) $$
are anti- holomorphic    sections of  the normal bundle of $\Sigma$.
\item  We have: 
\begin{equation} 
\left\{
\begin{array}{cc}\label{coeff}
E  |\gaussp_{L,x} |^2   = &   \frac{1}{2} |B|^2 + \delta    \\ 
E|\gaussp_{R,x} |^2   = & \frac{1}{2}|B|^2 - \delta  
\end{array}
\right.
\end{equation}
where  $ | B| ^2   :=  2 |B_{xx}|^2 +2 |B_{xy}|^2,\quad  \delta =  2\langle B_{xx}, \gaussp_L B_{xy}  \rangle$\\
\end{enumerate}
\end{prop}
 \begin{proof}
 
 \begin{enumerate} 
 \item  is clear.
 \item
 We  first show that the right Gauss map  $\gaussp_R$ is  anti-conformal.\\
Recall that $X : U \longrightarrow  \Sigma  \subset\quat $  be  a minimal   immersion  then 
for  isothermal coordinates   we have  $|X_x |= |X_y|, X_x \perp X_y $.
 If $X$ minimal   then $\Delta X = X_{xx}+X_{yy} = 0$ 
 

In order to prove 
$$\gaussp_{R,y} =   - \gaussp_{R,x}    \gaussp_{R} $$
one  first computes :
$$\gaussp_{R,x} =   (X_x^{-1} X_y)_x =X_x^{-1}(-X_{xx} X^{-1}_xX_y+ X_{xy}  )  =  X_x^{-1}(-X_{xx} \gaussp_{R} + X_{xy}  )  $$
and :
$$\gaussp_{R,y} =   (X_x^{-1} X_y)_y =X_x^{-1}(-X_{xy} X^{-1}_xX_y+ X_{yy}  ) = -X_x^{-1}(X_{xy} -X_{xx}  \gaussp_{R}  )  \gaussp_{R}  = - \gaussp_{R,x}    \gaussp_{R} .$$
Similarly one proves that 
\begin{equation}\label{conf1}
\gaussp_{L,y} = -\gaussp \gaussp_{L,x}.
\end{equation}
From which one deduces  the conformal property of the map :
\begin{equation}\label{conformite}
|\gaussp_{L,x} |= |\gaussp_{L,y} | , \langle \gaussp_{L,x} , \gaussp_{L,y} \rangle =0.
\end{equation}
\item 
A direct computation gives :
$$\gaussp_{L,x} = (X_y X_x^{-1})_x = (X_{xy} -\gaussp_L X_{xx})X_x^{-1} = -\gaussp_L ( X_{xx} + \gaussp_L X_{xy})X_x^{-1} .$$
Similarly, for the right Gauss map
$$\gaussp_{R,x} = (X_x^{-1}X_y) _x =X_x^{-1} (X_{xy} -X_{xx}\gaussp_{R}) = - X_x^{-1} (X_{xx} +X_{xy}\gaussp_{R})\gaussp_{R}.$$

Thus from \eqref{conf1}
$$\gaussp_{L,y} X_x= -( B_{xx} + \gaussp_L B_{xy}).$$

\item
We need the following lemma:
\begin{lem}
  $X_{xx} +\gaussp X_{xy}$ (resp.   $ X_{xx} + X_{xy}\gaussp_{R}$)   is a  holomorphic normal section of $\Sigma$   equal to $B_{xx} + \gaussp_L B_{xy}$
(resp.   to       $ B_{xx} + B_{xy}\gaussp_{R}$)      \end{lem}
 \begin{proof}
 
Let us prove first  that there is no component along  $X_x$ :
$$  \langle \gaussp X_{xy}, X_x \rangle =  -\langle X_{xy}, X_y \rangle = \langle X_{x}, X_{yy} \rangle =-\langle X_{x}, X_{xx} \rangle$$
Hence 
$$  \langle X_{xx} +\gaussp_{L} X_{xy}, X_x \rangle =  0$$
Similarly  $  \langle X_{xx} +\gaussp_{L} X_{xy}, X_y \rangle =  0$.
Hence  $X_{xx} +\gaussp_{L} X_{xy}$ is a normal section hence equal to $B_{xx} +\gaussp_{L} B_{xy}$.\\
In order to prove the holomorphy, we write  Codazzi equations in  the isothermal coordinates $x,y$:
\begin{equation}
\left\{
\begin{array}{cc}
B_{xx,y} -  B_{xy,x} &= 0 \\
B_{xy,y} -  B_{yy,x} &= 0
\end{array}
 \right.
 \end{equation}

Hence, using the  Codazzi equations:
$$\left( B_{xx} + \gaussp_L B_{xy}\right)_y = B_{xx,y} + \gaussp_{L,y} B_{xy}  +  \gaussp_{L} B_{xy,y} = B_{xy,x} - \gaussp_L \gaussp_{L,x} B_{xy}  +  \gaussp_{L} B_{yy,x} . $$
Since  $\Sigma$ is minimal :    $B_{xx}= -B_{yy}$  and
$$\left( B_{xx} + \gaussp_L B_{xy}\right)_y = -\gaussp_L \left(B_{xx} + \gaussp_L B_{xy}\right)_x .$$
\end{proof}

\begin{rem}
The Gauss map  $\gaussp_{L, .}$ is of the form  $ v.w^{-1} $ where $v\in N\Sigma$ and $w\in T\Sigma$ from which it is clear that $\gaussp_{L,.} (T\Sigma )= N\Sigma$
 and  from which  one deduces also that  $\gaussp_{L,.} (N\Sigma )= T\Sigma$
 
\end{rem}

\item 
Let $E  = |X_x|^2 = |X_y|^2$ then 
\begin{equation}\label{getb}
|\gaussp_{L,x}|^2 = \frac{1}{E} | X_{xx} +\gaussp_L X_{xy}|^2 \ {\rm and} \    |\gaussp_{R,x}|^2 = \frac{1}{E} | X_{xx} + X_{xy}\gaussp_{R}|^2.
\end{equation}
Thus
$$E|\gaussp_{L,x} |^2 = |B_{xx} + \gaussp_L B_{xy} |^2 =  |B_{xx}|^2 + |B_{xy}|^2 + 2\langle B_{xx}, \gaussp_LB_{xy}   \rangle.$$
And 
$$E|\gaussp_{R,x} |^2 =   |B_{xx}|^2 + |B_{xy}|^2 + 2\langle B_{xx}, B_{xy}\gaussp_R   \rangle.$$
As
$\langle B_{xx}, B_{xy}\gaussp_R   \rangle = - \langle B_{xx}, \gaussp_LB_{xy}   \rangle$ we obtain :
\begin{equation}\label{equa}
\left\{
\begin{array}{cc}
  E|\gaussp_{L,x} |^2   = &  \frac{1}{2} |B|^2 + \delta    \\ 
E|\gaussp_{R,x} |^2   = &   \frac{1}{2}  |B|^2 - \delta  
\end{array}
\right.
\end{equation}
where  $ |B|^2  :=  2 |B_{xx}|^2 +2 |B_{xy}|^2$ and $ \delta =  2\langle B_{xx}, \gaussp_L B_{xy}  \rangle$.

\end{enumerate}
\end{proof}
\begin{rem}
 One can express the curvatures  of the tangent and normal bundle of a minimal surface in terms of the derivatives of the  left and right  Gauss maps in isothermal coordinates.\\
  From the Gauss equation  we deduce that the curvature tensor of the tangent bundle  is equal to \\
$R(X_x,X_y,X_y,X_x) := R_{xyyx} = \langle B_{xx},B_{yy}\rangle -  \langle B_{xy},B_{yx}\rangle = -|B_{xx}|^2-|B_{xy}|^2 .$\\
Hence  the tangent Gaussian curvature  computed in the chosen  isothermal coordinates is equal to:
$K^T := \frac{R_{xyyx} }{|X_x|^2|X_y|^2-\langle X_x,X_y\rangle^2} = -  \frac{|B|^2 }{2E^2}$.
 Similarly, the  Ricci equation gives for the normal curvature in terms of $B$:
 $\langle R^N(X_x,X_y)\xi, \eta\rangle =   \langle \left[  A_\xi, A_\eta\right] X_x,X_y \rangle$
where $\xi,\eta$  are normal vectors  and $A_\xi$ is the shape operator of the the second form $B$ in the direction $\xi$ so that $\langle A_\xi X , Y\rangle =  \langle B(X , Y), \xi \rangle$.
Choose   $\xi =e_1,\eta = e_2$  such that $\{X_x,X_y, e_1,e_2\}$ is a conformal frame in $\quat$. Let  $B^i_{ab} := \langle B_{ab},  e_i\rangle$ ( $i =1, 2,\quad a,b = x,y$). Then 
 $K^N :=  \frac{R^N_{xy12} }{E^2}   =\frac{\langle R^N(X_x,X_y)e_1, e_2\rangle}{E^2} = - \frac{2}{E}\left(  B^1_{xx} B^2_{xy} -B^2_{xx}B^2_{xy} \right)=  \frac{2}{E^2}\langle B_{xx}, \gaussp_L B_{xy}  \rangle
 = \frac{1}{E^2}\delta$.
From  equations \eqref{equa}  we deduce that
\begin{equation} \label{courbure}
\left\{
\begin{array}{cc}
K^T   & =   - \frac{1}{2E}   \left(  |\gaussp_{L,x} |^2+ |\gaussp_{R,x} |^2  \right)  \\ 
 K^N    & =      - \frac{1}{2E}   \left(  |\gaussp_{L,x} |^2- |\gaussp_{R,x} |^2  \right)
\end{array}
\right.
\end{equation}
We see in particular that  $|K^T|\geq |K^N|$.
 \end{rem}
\section{\normalfont Proof of Theorem 2 and corollaries}

Let us recall \cite{BDC}'s stability condition for minimal surfaces in $\mathbb{R}^3$.   
\begin{thm} \cite{BDC} Let $\Sigma$ be a minimal surface. 
If the spherical area   $g(\Sigma)\subset \sd $  in  the round  sphere of radius one   is   smaller than $2\pi$ then  $\Sigma$ is stable.
\end{thm}

We first  follow the proof of    \cite{BDC}  applied to the right and left Gauss maps.
\begin{proof} 
 Let $ X : U \longrightarrow \Sigma \subset \mathbb{R}^4$ be   the minimal immersion. 
 Suppose the first Dirichlet  eigenvalue    of the  spherical   image $D :=\gaussp_L (U)$  is  $\lambda$;   then  from Rayleigh's formula
 \begin{equation}\label{inegalite}
 \int_D |\nabla f|^2_{\sd}da_{\sd } -  \lambda\int_D f^2da_{\sd} \geq 0
 \end{equation}
 for all compact support functions  $f : D \longrightarrow \mathbb{R}$.\\
 From the conformality of $\gaussp_L$ and from  equations  \eqref{conformite},  the left Gauss map $\gaussp_L$  induces on $U$ 
a  spherical area and  a spherical pseudo-metric  which in terms of the pulled-back metric element  $ds^2_\Sigma$ or area element $da_\Sigma$  of $\Sigma$ on $U$ are  -using equations \eqref{courbure}-   given by  
\begin{equation}
\gaussp_L^*(da_{\sd} ) = \frac{|\gaussp_{L,x}|^2}{E} da_\Sigma = -(K^T+K^N)  da_\Sigma  \quad  \gaussp_L^*(ds^2_{\sd} ) =-(K^T+K^N) ds^2_\Sigma.
\end{equation}
On a subdomain $U'\subset U$ where  $\gaussp_L$ is a diffeomorphism,  $ |\gaussp_{L,x}|^2 ds^2$ is a metric element  and 
$|\nabla (f\circ\gaussp_L) |_\Sigma^2 da_\Sigma =  |\nabla f|^2_{\sd} da_{\sd} .$ Thus 
 \begin{equation}\label{inegalitebis}
 \int_{\gaussp_L(U')}\left(  |\nabla f|^2_{\sd} - \lambda f^2\right) da_{\sd}   =  \int_{U'} \left(  |\nabla (f\circ\gaussp_L)|_\Sigma^2 + \lambda( K^T+K^N)(f\circ\gaussp_L)^2\right) da_\Sigma. 
 \end{equation}
 
 The functional  defined by the RHS of \eqref{inegalitebis} extends to $U$  and   \cite{BDC2} show  that   for all compact support functions  $f : U \longrightarrow \mathbb{R}$
  \begin{equation}\label{inegalite2}
\int_{U} \left(  |\nabla f|_\Sigma^2 + \lambda( K^T+K^N)f^2\right) da_\Sigma \geq 0 \end{equation}
( 
 if the LHS were negative for a compact support function $f$ on $U$ , then  the function on $D$ defined by  $\bar f(x) := \sum_{\gaussp_L(y) = x} f(y)$ would not satisfy inequality 
\eqref{inegalite}; for more details see \cite{BDC2}).\\
The same considerations for the right Gauss map yields also - for the first Dirichlet eigenvalue $\mu$ of the spherical domain $\gaussp_R(U)$:
$$\int_U |\nabla f|_\Sigma^2 da_\Sigma + \mu\int_U(K^T-K^N)f^2da_\Sigma $$
and  combining linearly  the former  two inequalities  and applying equations \eqref{courbure}:
\begin{equation}\label{stable}
\int_U |\nabla f|_\Sigma^2 da_\Sigma +\frac{2\lambda\mu}{\lambda+\mu}\int_U K^T f^2da_\Sigma \geq 0
\end{equation}
On the other hand  recall that   the surface $\Sigma$ is stable  in $\mathbb{R}^4$  if for  any normal section $\xi$  with compact support  which vanishes at the boundary  
 ( cf. for example \cite{L}) :
\begin{equation}\label{stableNormalante}
\int_U\left(  |\nabla^N \xi |_\Sigma^2   - \sum_{i,j= 1,2 }   \langle  B_{ij},\xi   \rangle^2 \right)   da_\Sigma  \geq 0
\end{equation}
where $\nabla^N$ is the induced connection on the normal bundle of $\Sigma$ and where $B_{ij}$ are the coordinates of the second form wrt to an orthonornal basis of $T_p\Sigma$ - for example  wrt
$\{\frac{X_x}{\sqrt{E}}, \frac{X_y}{\sqrt{E}}\}$. Hence  :
\begin{equation}\label{stableNormal}
\int_U\left(  |\nabla^N \xi |_\Sigma^2   - \sum_{i,j= x,y } \frac{1}{E^2}  \langle  B_{ij},\xi   \rangle^2 \right)   da_\Sigma  \geq 0
\end{equation}

Since $$  \frac{1}{E^2}   \sum_{i,j= x,y}   \langle  B_{ij},\xi \rangle^2 \leq -2K^T |\xi|^2 ,$$
the LHS of inequality \eqref{stableNormal} is larger than 
 
\begin{equation}\label{stable10}
\int_U |\nabla^N \xi|^2_\Sigma da_\Sigma +2 \int_U K^T |\xi|^2 da_ \Sigma.
\end{equation}
Choose    $\xi = f \zeta$ where $f$ is of compact support and where $\zeta$ is a unitary vector field. Using the same notations as in Proposition \ref{propSurG}
then 
\begin{equation}\label{energie}
\int_U |\nabla^N \xi |_\Sigma^2 da_\Sigma = \int_U |\nabla^N \zeta |_\Sigma^2 f^2  da_\Sigma  + \int |\nabla f |_\Sigma^2 da_\Sigma.
\end{equation}
Replacing in  \ref{stable10} the first term  by \ref{energie} and using \ref{stable}, we obtain 
\begin{equation}\label{equaInequa}
\int |\nabla^N \zeta |_\Sigma^2 f^2  ds +\left( 2 -  \frac{2\lambda\mu}{\lambda+\mu}  \right) \int  K^T  f^2ds \end{equation}
   which is positive if 
\begin{equation}\label{coeffL}
 \frac{\lambda\mu}{\lambda+\mu} > 1
 \end{equation}
that is    if the harmonic mean  of  the eigenvalues $\lambda$ and $\mu$ is larger than 2.
 \end{proof}
 \subsection{\normalfont Stability in terms of the  spherical area of the Gauss maps} \label{conjecture}
One can deduce from the isoperimetric  inequality in the round sphere  a lowerbound for the first Dirichlet eigenvalue of the spherical domain $D$ of area $A$ as in \cite{BDC2}:
\begin{equation}\label{iso}
\lambda_1(A) \geq \frac{2(4\pi-A)}{A} = 2( \frac{1}{a}- 1)
 \end{equation}
  where $a := \frac{A}{4\pi}$ is the proportionate  area of $D $. Hence 
Inequality \eqref{coeffL} is true if 
$$ \frac{1}{\lambda} +   \frac{1}{\mu}  \leq \frac{a}{2( 1-a ) }    +   \frac{b}{2( 1-b ) }  \leq 1$$
Therefore :\label{page:hypequi1}
\begin{prop}\label{harmonique}\label{hypequi}
If   the proportionate area $a = \frac{|\gaussp_L(\Sigma) |}{2\pi} $  and $b = \frac{|\gaussp_L(\Sigma) |}{2\pi} $ satisfy \\
$a \leq \frac{2}{3},\ b \leq \frac{2-3a}{3-4a}  $  then $\Sigma$ is stable.
\end{prop}
  The    set  $\{(a,b) :  0 \leq  b \leq \frac{2-3a}{3-4a}   \  0\leq a,b \leq  \frac{2}{3} \}$ is  a  domain  in the square $[0,1]\times[0,1]$  bounded  by the edges $x=0$, $y=0$ and 
 the  equilateral hyperbola passing through the points $(2/3,0),(0,2/3), (1/ 2, 1/2    )$  which corresponds to domain 3 in Figure \ref{fig} and domain 1 in Figure \ref{fig4}. \\
 The lowerbound in inequality \eqref{iso} is not optimal. Using Sato's second approximation in \cite{S}, we obtain a larger domain of stability bounded by the equilateral hyperbola
 passing  through the points $(.737,0),(0,.737), (1/ 2, 1/2    )$ (illustrated by  domain II of  Figure \ref{fig4}). In a similar fashion the stability domain of      Corollary \ref{cor1}
 can be enlarged.
 
\begin{figure}[h!]\label{fig4}
\begin{center}
\includegraphics[scale=0.4]{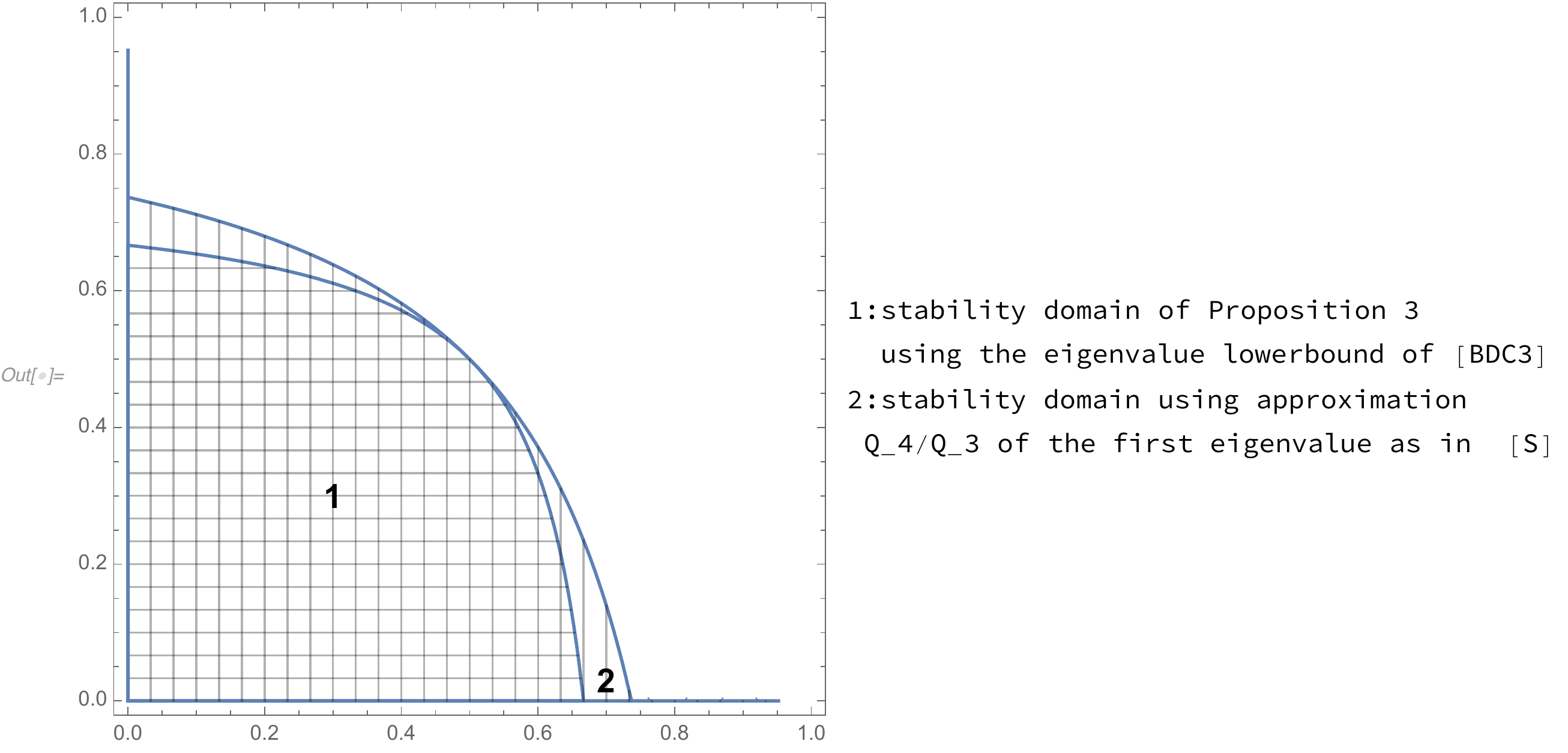} 
\caption{$ \Sigma \mapsto \left(\frac{ |\gaussp_L\left(\Sigma\right)|}{2\pi},\frac{ |\gaussp_R\left(\Sigma\right)|}{2\pi} \right)$  }
\end{center}
\end{figure} 
\newpage
  \subsection{\normalfont Stability domain for  minimal surfaces with flat normal bundle}

 Note from  \eqref{coeff} and \eqref{courbure}     that  $|\gaussp_{L,x}|^2 $ and  $|\gaussp_{R,x} |^2 $    are identical  iff  $\delta = K^N=0$. The normal bundle is then  flat
   which means $B_{xx}$ and $B_{xy}$ are colinear from the  expression of $\delta$ in equation \eqref{coeff}. One deduces that  the surface  
lies  in a hyperplane (see for example \cite{A}) Hence  minimal surfaces have a flat normal bundle iff they lie in $\mathbb{R}^3$.\\ 
Furthermore  the Weierstrass representation  of such  minimal surface  can be chosen of the form $X = (e+\bar f, g + \bar g)$ and in the coordinate system 
of Section \ref{Weier} we see that the first eigenvalues  $\lambda$ an $\mu$ of $\gaussp_L$ and $\gaussp_R$ are identical.
  Stability is  then obtained if the  first eigenvalue $\lambda$  of the right or left spherical domain is larger than 2. 
The corresponding domain of stability   is the diagonal of the square represented by the domain 2 of Figure \ref{fig} 
 \subsubsection{\normalfont  Unstability  } 
If  the surface lies in a hyperplane then there is a constant normal vector field $\zeta$.  Then plug  $\xi= f\zeta$ into LHS of \ref{stableNormal} and using  equality \ref{energie} we obtain  :
\begin{equation}\label{stable1}
\int_U |\nabla f |_\Sigma^2   da_\Sigma   -\int_U  \sum_{i,j= x,y} \frac{1}{E^2}  \langle  B_{ij},\zeta \rangle^2   f^2 da_\Sigma . 
\end{equation}
 Then  replace  in expression \eqref{stable1} the normal field   $\zeta $ by     the normal fields $e^{\gaussp t}\zeta =( \cos t + \gaussp \sin t)\zeta $  and average wrt  $  \ t\in [0,2\pi]$. The second term  of   \eqref{stable} becomes
 $$ \int_\Sigma \frac{1}{E^2}( |B_{xx}|^2 + |B_{xy}|^2)f^2da_\Sigma.  $$
 Choose $f\in C_0(U)$ such that  
 $$ \int |\nabla f|^2 ds = \frac{\lambda}{ 2}\int  \frac{1}{E^2}(   |B_{xx}|^2+     |B_{xy}|^2       ) f^2ds .$$
 ( the left and right eigenvalues are equal ($\lambda=\mu$) hence $ \frac{2\lambda\mu}{\lambda+\mu} =\lambda$ in \eqref{stable}.\\
Consequently \eqref{stable1}  is negative if $\lambda<2$.
Hence 
\begin{cor} A   minimal  surface of $\mathbb{R}^4$  with flat normal bundle is stable if  the first eigenvalue $\lambda$  of  the left (or right) spherical domain $\gaussp_L (\Sigma)$   satisfies  $\lambda >2$ 
  and unstable if  $\lambda <2$.

\end{cor}
 \subsubsection{\normalfont  An example } 
 Following the notations defined in  Section \ref{Weier}, define 
  on the slab $U= \{z \in \mathbb{C} :  0 \leq \Im(z) < 2\pi \}$    a minimal immersion $X: U \longrightarrow \mathbb{R}^4$  of Weierstrass representation:
$ e = e^{-z}, f= e^z, g = z = h $ where $X(z) = (e(z)+\bar f(z), g(z) +\bar g(z) )=   2 ( \cosh x e^{iy} +xj) $.\\
From \eqref{r}  $g_L= -g_R$ and 

$X_x =  e' + \bar f' + (g'+ \bar g')j = 2 ( \sinh x e^{iy} +j) ,   X_y =  i ( e' - \bar f'+(g'- \bar g')j)=  2 i \cosh x e^{iy}  $ 
$$\gaussp_L = \frac{1}{E} X_y\bar X_x =  \frac{i}{E} ( |e^z|^2 - |e^{-z}|^2 + 2(e^{-z} +\overline{e^{z}} )j) $$
$$\gaussp_R = \frac{1}{E} X_y\bar X_x =  \frac{i}{E} ( |e^z|^2 - |e^{-z}|^2 + 2(e^{z} +\overline{e^{-z}} )j) $$
$$\gaussp_R -\gaussp_L = \alpha k.$$
Hence
$$  \langle \gaussp_L -\gaussp_R, X \rangle = 0.$$
The catenoid $X(U)$ is contained in $span( 1,I,J)$ and $\xi=  \gaussp_L -\gaussp_R $  is a   normal section of the normal bundle of $X(U)$.
 \section{\normalfont  The area of  a holomorphic curve in $\mathbb{S}^2 \times \mathbb{S}^2$  } 
K\"ahlerian geometry  on  $\mathbb{S}^2 \times \mathbb{S}^2$  provided with the product of complex structures on each   $\mathbb{S}^2$   shows that 
 the area of a holomorphic curve in $\mathbb{S}^2 \times \mathbb{S}^2$  
is the sum of the projected area on each sphere- counted with multiplicity (see remark \ref{remarkone}). More precisely 
\begin{lem} \label{kaehler}
Let $G$ be a holomorphic curve in   $\mathbb{S}^2 \times \mathbb{S}^2$  and let $\pi_i, i = 1, 2$ be the projections  maps of $G$ on each sphere, 
then 
$$ | G| = |\pi_1(G)|_{\pi_1^*(\sd)}+   |\pi_2(G) |_{\pi_2^*(\sd) }$$
where  $ |\pi_i(G)|_{\pi_i^*(\sd)}$ denotes the spherical area on $G$   pulled-back  by $\pi_i$, $i =1,2$.
  \end{lem}

\begin{proof}
The K\"ahler form  of $ \mathbb{S}^2\times  \mathbb{S}^2$  ( provided with  the metric which is the  direct product  of round sphere metrics)  is the sum of the Kahler form on each factor.
Choose  local isothermal coordinates  of the curve $G$; then the K\"ahler-area  form -  $\omega$ in the ambient space is equal to :
$$\omega(X_x,X_y) = \langle  iX_x, X_y     \rangle   =  \langle  iX^1_x+iX^2_x, X^1_y  +   X^2_y  \rangle.   $$
Hence
$$\omega(X_x,X_y) =   \langle  iX^1_x , X^1_y    \rangle  +    \langle  iX^2_x , X^2_y    \rangle=  \omega_{\mathbb{S}^2}(X^1_x,X^1_y) +   \omega_{\mathbb{S}^2}(X^2_x,X^2_y).  $$
And
$$|G| = \int_G \omega = \int_G \pi_1^*\omega_{\mathbb{S}^2} + \int_G \pi_2^*\omega_{\mathbb{S}^2}$$
\end{proof}
Comparing the projected area on each sphere  with the projected  area counted with multiplicity, we have 
\begin{cor}\label{aires} Let $\Sigma$ be a  minimal surface in $\mathbb{R}^4$ then 
$$|g( \Sigma)| \geq  |\gaussp_L( \Sigma)|  +  |\gaussp_L( \Sigma)| $$
where $g = (\gaussp_L,\gaussp_R)$ is the Gauss map of $\Sigma.$
\end{cor}
 Hence we see that an upperbound on the sum $ |\gaussp_L( \Sigma)|  +  |\gaussp_L( \Sigma)| $ as in Corollary \ref{cor1} 
  does not impose an upperbound on $|g( \Sigma)| $
 as in Theorem \ref{thm4}. 
Furthermore there is no a priori upperbound for $ |g(\Sigma)|$.
    \begin{prop}\label{nobound}
There is a family of minimal surfaces $\Sigma_n$ such that $\lim |g(\Sigma_n)| = +\infty$.
\end{prop}	
\begin{proof}
Consider the minimal surface $\Sigma_n$ conformal to $\mathbb{C}$ and given by the Weierstrass representation ( see  expression \ref{Weier} in next Section)
	$$e'=z^n\  \ \ f'=1\  \ \ g'=z^p\  \ \ h'=-z^{q}$$
	with $p+q=n$, $(p,q)=1$.\\
The  holomorphic Gauss maps - defined as in equations \eqref{r} , are $g_L=z^p$ and $g_R=-z^{q}$. 
Let us show that the map $(g_L,g_R)$ is injective.
Let $(z_1,z_2)$ be a double point of $(g_L,g_R)$. Then $z_1^p=z_2^p$ and  $z_1^q=z_2^q$. Since $(p,q)=1$, this implies that $z_1=z_2$. 
Thus  from Lemma \ref{kaehler}, $|g(\Sigma)|$ is equal to $2\pi p + 2\pi q =2\pi n$ .
\end{proof}
Still,  the area  of the Gauss map image of  complex curves is  bounded  above by $2\pi$.
\section{\normalfont Proof of Theorem \ref{thm4}}

The set of minimal surfaces in $\mathbb{R}^4$ are locally easy to describe.  The data defining a minimal  immersion of a disk is a quadruple 
of holomorphic function $(e,f,g,h)$  on $U$ 
such that 
\begin{equation} \label{Weier2}  
e'f'+g'h'=0. 
\end{equation}

The minimal immersion is  then given by the map 
\begin{equation}\label{Weier}
X: \left(
\begin{array}{ll}
 U &\longrightarrow \quat \\
    z &\mapsto    e(z)+\bar f(z) +( g(z)+\bar h(z))J
    \end{array}
\right)
\end{equation}
the coordinates  ({\it Weierstrass coordinates})  of the immersion are automatically harmonic.
In this algebraic setting, the conformality  condition of $X$  - given by equation \eqref{Weier2} - is best understood in the complexified ambient space.\\
We denote the complexified quaternions by: $\quatc = \quat \otimes_\mathbb{R}\mathbb{C}$.
The Euclidean scalar product  $\langle \cdot, \cdot \rangle$ on $\quat$ extends to $\quatc$   either as  the complexified scalar product 
 $\langle \cdot, \cdot \rangle_{\mathbb{C}}$ or  as  the hermitian metric  $( \cdot, \cdot )_{\mathbb{C}}$.\\
The immersion  $X$ is conformal  wrt  $z\in U $ iff  the complexified vector  $(X_x-iX_y)(z)$ is a null vector i.e. iff $\langle X_x-iX_y, X_x-iX_y\rangle_{\mathbb{C}} =0$ . The tangent Gauss map 
can be  identified to 
\begin{equation}
G': 
\left( 
\begin{array}{ll}
 U & \longrightarrow  Q'_2 \subset P(\quatc)\\
 z & \mapsto \left[X_x(z)-iX_y(z) \right]
 \end{array}
\right)
\end{equation}
where $Q'_2$ is  the   complex 2-dimensional quadric of the complex 3-dimensional projective space  $P(\quatc)$  
- identified to  the Grassmanian $G^+(2,4)$ -  which  is defined  as the zeroes of  the  following quadratic form 
$$Q'_2 := \{ [Z ]\in P(\quatc) :    q'_2(Z) :=  z_1^2+z_2^2 + z_3^2+z_4^2 =0\ {\rm for} \   Z = z_1 +z_2I+z_3J+z_4K \} .$$
If we plug the holomorphic coordinates of \eqref{Weier} into $G'$, we obtain  
  an equivalent Gauss map 
  \begin{equation}
G: 
\left( 
\begin{array}{ll}
 U & \longrightarrow  Q_2 \subset P(\quatc)\\
 z & \mapsto \left[e'(z),f'(z),g'(z),h'(z)\right]
 \end{array}
\right)
\end{equation}
where 
$$Q_2 := \{ [Z ]\in P(\quatc) :    q'_2(Z) :=  z_1 z_2  + z_3 z_4  =0\ {\rm for} \   Z = z_1 +z_2I+z_3J+z_4K \} .$$
(The Weierstrass coordinates whose Gauss map takes its values in  $Q'_2$ appeared in \cite{E} and  
those adapted to a parametrization of $Q_2$ appeared   in \cite{MW} ).
  The passage from this algebraic description 
of the Gauss map to the   geometric one is summarized in the following diagrams;  the  second diagram gives the expressions of the maps of the first diagram.
\begin{figure}[h!]\label{diagram}

\begin{tikzcd}
\quat  \arrow[r, leftarrow, black, "X"]&U  \arrow[r, black, "G"] \arrow[d, "\gaussp",black]   \arrow[rd, dashrightarrow, " \tilde g",black] &  Q_2  \arrow[r, hookrightarrow  , "" ] 
 \arrow[d, dashrightarrow,  "p",black]  \arrow[rd,leftrightarrow , "Segre\  map" ] &   \mathbb{P}^3 \\
 &   \mathbb{S}^2 \times  \mathbb{S}^2   \arrow[r,leftarrow," \sigma \times \sigma " ] &  \mathbb{C} \times  \mathbb{C}  \arrow[r,dashleftarrow,"" ] &  
    \mathbb{P}^1 \times  \mathbb{P}^1      
    \end{tikzcd}
\\
\begin{tikzcd}
{\scriptstyle e(z)+\bar f(z) +( g(z)+\bar h (z) )j} \arrow[r, leftarrow, black, "X"]
& {\scriptstyle z =x+i y}\arrow[r, black, "G"] \arrow[d, "\gaussp",black]   \arrow[rd,dashrightarrow, "\tilde g",black]  
& {\scriptstyle   [e'(z),f'(z),g'(z),h'(z)] }  \arrow[r, hookrightarrow  , "" ]  \arrow[d,dashrightarrow, "p",black] &  
{\scriptstyle [z_1z_3,z_2z_4,-z_2z_3,z_1z_4] } \arrow[d,leftarrow , "\sigma_e" ] \\
 &  \scriptstyle{\left(\gaussp_L\left(z\right),  \gaussp_R\left(z\right) \right) } \arrow[r,leftarrow," \sigma \times \sigma   " ] & 
  \scriptstyle{(\frac{z_1}{z_2} , \frac{z_3}{z_4}    ) }  \arrow[r,dashleftarrow, "" ] &      \scriptstyle{( [z_1,z_2],  [z_3,z_4]  ) } 
\end{tikzcd}
\caption{ Map diagram of the Gauss map in Weierstrass coordinates  }
\end{figure}
\vskip .2 in 
where :  $G(z) = X_x - i X_y $,  $\sigma: \mathbb{C} \longrightarrow  \mathbb{S}^2$ is  the  stereographic projection, 
$ \gaussp  :=(\gaussp_L,\gaussp_R) , \tilde g :=(g_L,g_R) $, $p([z_1,z_2,z_3,z_4]) = (\frac{z_1}{z_4}, \frac{z_4}{z_2})$.
The dotted arrows represent rational maps. The holomorphic Weierstrass data of the Gauss maps are then:


\begin{equation}  \label{r} 
g_L (z) = \frac{z_1}{z_2}  = \frac{e'(z)}{h'(z)} = -  \frac{g'(z)}{f'(z)} ,\quad   g_R(z) = \frac{z_3}{z_4}  =- \frac{e'(z)}{g'(z)}=  \frac{h'(z)}{f'(z)}.
\end{equation}

Moreover  $P(\quatc)$  is provided with the Fubini metric $ds^2_F$ so that the projection  \\
 $\pi: (\quatc\setminus\{0\} , ( \cdot, \cdot )_{\mathbb{C}})   
 \longrightarrow  (P(\quatc), ds^2_F)$
is a Riemannian submersion.  Thus the   unitary group $U(4)$ action on  $\quatc$  descends to   isometries  on $P(\quatc)$.

  
The quadrics  $Q_2$ and $Q'_2$ are each   provided with the  metric induced by the  ambient Fubini metric.  
 One can easily check that  $Q_2$    is isometric to   $Q'_2$   via the  group element $A\in U(4)$
defined by :
\begin{equation}
A= \left( 
\begin{array} {cc}
 A_1 &  O  \\
 O  & A_1 
 \end{array}
 \right) \ {\rm where} \ 
 A_1=  \frac{1}{\sqrt{2} }\left( 
\begin{array}{cc} 
 1 &  i  \\
 1  & -i 
 \end{array}
 \right)
\end{equation}
i.e.  : $  \forall v\in \quatc\    \quad q_2(Av) = q'_2(v)$.\\
Moreover it is clear that $Q'_2$ is stable by  $O(4,\mathbb{C})$. We then deduce that    $Q_2$ is stable  by  
the conjugate group  $A\cdot O(4,\mathbb{C})\cdot  A^{-1} $. In particular 
$Q_2$ is stabilized by  the subgroup \\
 $ A\cdot SO(4,\mathbb{R})\cdot A^{-1}  \subset SU(4) $.    
We introduce now   the  following permutation of the left and right sphere  : 

\begin{equation}\label{swich}
\pi : \left(
 \begin{array}{ll}
   \mathbb{P}^1 \times     \mathbb{P}^1  & \longrightarrow    \mathbb{P}^1 \times     \mathbb{P}^1    \\
    ( [a,b],[c,d] )  & \mapsto    ([d,c],  [a,b] ) 
 \end{array}
 \right)
\end{equation}
One checks that $\pi = \sigma_e^{-1}\circ S \circ\sigma_e$ where   $\sigma_e$ is  defined in the diagram of Figure \ref{diagram}  and $S$ is the  ambient linear map  $S$ of $\quatc$ :

 \begin{equation}
S= \left( 
\begin{array} {cc}
O &  S_1  \\
 S_2  & O 
 \end{array}
 \right), \ {\rm where} \ 
 S_1=  \left( 
\begin{array}{cc} 
 0 &  1  \\
- 1  & 0 
 \end{array}
 \right)  \ {\rm and} \ 
 S_2=   \left( 
\begin{array}{cc} 
 -1 &  0  \\
0  & 1 
 \end{array}
 \right). \end{equation}
By construction $Q_2$ is stable by the action of $S$ and by its expression one sees that  $S \in SO(4,\mathbb{R})\subset SU(4)$.
\begin{lem}\label{isotopy} There exists an isotopy $H : ([ 0,1] \longrightarrow \mathcal{J} :=U(4) \cap A\cdot  SO(4,\mathbb{C})\cdot  A^{-1})$ 
of isometries acting on $Q_2\subset P(\quatc)$   such that $H(0)= Id$ and $H(1) = S$.
\end{lem}
\begin{proof}
   $Id$ and $S$ belong to $\mathcal{J}$ which is path-connected.
Indeed, its conjugate  by $ A\in U(4)$ equals
$A^{-1} \mathcal{J} A= U(4) \cap SO(4,\mathbb{C}) = SO(4,\mathbb{R}) $ which is  path-connected.
\end{proof}

\subsection{\normalfont Deformations by associate minimal surfaces}
We start with  a minimal surface  whose Gauss map  in  Weierstrass coordinates -following the notations of Section \ref{Weier} - are  of the form 
\begin{equation}
G: \left( 
 \begin{array}{ll}
U & \longrightarrow    G(U) \subset Q_2 \subset   P(\quat \otimes\mathbb{C})\\
  z  &\longrightarrow    [e'(z),f'(z),g'(z),h'(z)]
 \end{array}
 \right)
\end{equation}
NB. If the Riemann surface $U$ is not simply-connected then we replace it by  its universal cover; the image by the Gauss map of $U$ will not be affected.\\
From  Lemma \ref{isotopy} , there exists a  continuous path  $\gamma : [0,1] \longrightarrow   \mathcal{J} $ with $\gamma_0= Id$ and $ \gamma_1 = S$.\\
This  generates  a continuous  family of maps   $ G_t  := \gamma_t . G$.  Since $\gamma_t \in \mathcal{J}$  then  $\gamma_t G(U)\subset Q_2$. 
  In terms of Weierstrass coordinates, let     $ G_t  := [e_{t,1},e_{t,2},e_{t,3},e_{t,4}]$ 
  then  $e_{t,1}.e_{t,2}+ e_{t,3}e_{t,4} =0$. 
The   $G_t$ are then the  Gauss maps of a family of new minimal surfaces  - so called   associate minimal surfaces  to $\Sigma$.
\begin{equation}\label{famille}
X_t  : \left(
 \begin{array}{ll}
U & \longrightarrow   \Sigma_t \subset  \quat    \\
  z  &\mapsto     \int_z e_{1,t} + \overline{ \int_z e_{2,t}}  + ( \int_z e_{3,t} + \overline{ \int_z e_{4,t}} )J
 \end{array}
 \right)
\end{equation}
where $\Sigma_0 = \Sigma$ and  by  the defintion  of   \eqref{swich} such that  the Gauss map of $\Sigma_1$ is given by :
\begin{equation}\label{perm}
(\gaussp_{L,1}, \gaussp_{R,1}) = (\sigma_a\circ \gaussp_{R}, \gaussp_{L})  
\end{equation}
where $\sigma_a$ is the antipodal symmetry.\\
Let us describe some invariants  of this family of minimal surfaces.  
\begin{lem}\label{propAsso}
For any minimal surface  $\Sigma$ there is an associate family of   minimal surfaces  to $\{\Sigma_t\}_{t\in [0,1]} $   defined by  \eqref{famille}  
such that  the following   conditions are satisfied: \\
 $ \forall p \in U\ {\rm and} \  \forall   t\in [0,1] :$ 
\begin{enumerate} 
\item $\Sigma_t$ is  locally isometric to $\Sigma_0 = \Sigma$: 
$$  ds^2_t(p) := \lambda^2_t (p)  |dz|^2 = \lambda^2 (p)  |dz|^2  $$ 
\item The tangent curvature is invariant by deformation :  
$$    K_t^T(p) = K^T(p)  $$
\item The images of the Gauss maps have the same area :  
$$| g_t(\Sigma_t) | = | g(\Sigma) | $$
\item If  the operator  $\Delta_{\Sigma_{t_0}} - 2 K^T_{t_0 }$  is positive for some $ t_0\in [0,1]$  then  $\Delta_{\Sigma_t} - 2 K^T_t $ is positive for all $t\in[0,1]$.

\end{enumerate}
\end{lem}
\begin{proof}
\begin{enumerate}
\item We consider the deformation of $\Sigma$ defined in \eqref{famille}.In local isothermal coordinates $z$ , in a neighborhood of  some $p\in U$, $ds^2 = \lambda |dz|^2$.  where $\lambda =  |e'|^2+ |f'|^2 +|g'|^2+|h'|^2 = (G,G)$
where $( , )$ is the hermitian metric in $\quatc$. 
As $\gamma_t \in U(n)$, the metrics  $ ds_t^2 = \lambda_t |dz|^2 $ of $\Sigma_t$ are all locally isometric to $\Sigma$  since
$$\lambda_t  = (G_t,G_t) = |e_{1,t}|^2 + |e_{2,t}|^2 + |e_{1,2}|^2 + |e_{3,t}|^2  +  |e_{4,t}|^2=   (\gamma_tG,\gamma_tG) =(G,G).$$
\item The metric is invariant, so is the tangent  curvature.\\
\item   $|g_t(\Sigma)| = |\gamma_tg(\Sigma)| = |g(\Sigma)|$ \\
\item  If  for some $t_0$:
\begin{equation}\label{fin}
 \int_U|\nabla \phi|_{\Sigma_{t_0}} ^2da_{\Sigma_{t_0}}  \geq -2\int_U  K_{t_0}^T \phi^2 da_{\Sigma_{t_0}}  \ \forall \phi \in C_0(U),
 \end{equation}
 then, as  the  $\Sigma_t$ are all isometric,  $|\nabla \phi|_{\Sigma_{t_0}} ^2 =|\nabla \phi|_{\Sigma_{t}} ^2$ and $da_{\Sigma_{t_0}}=da_{\Sigma_{t}}$ and by the second point,
the  stability inequality  \eqref{fin} is true  for any $t\in[0,1]$.
\end{enumerate}
 \end{proof}
\begin{figure}[h!]\label{fig2}
\begin{center}
\vskip -2 in
\includegraphics[scale=0.3]{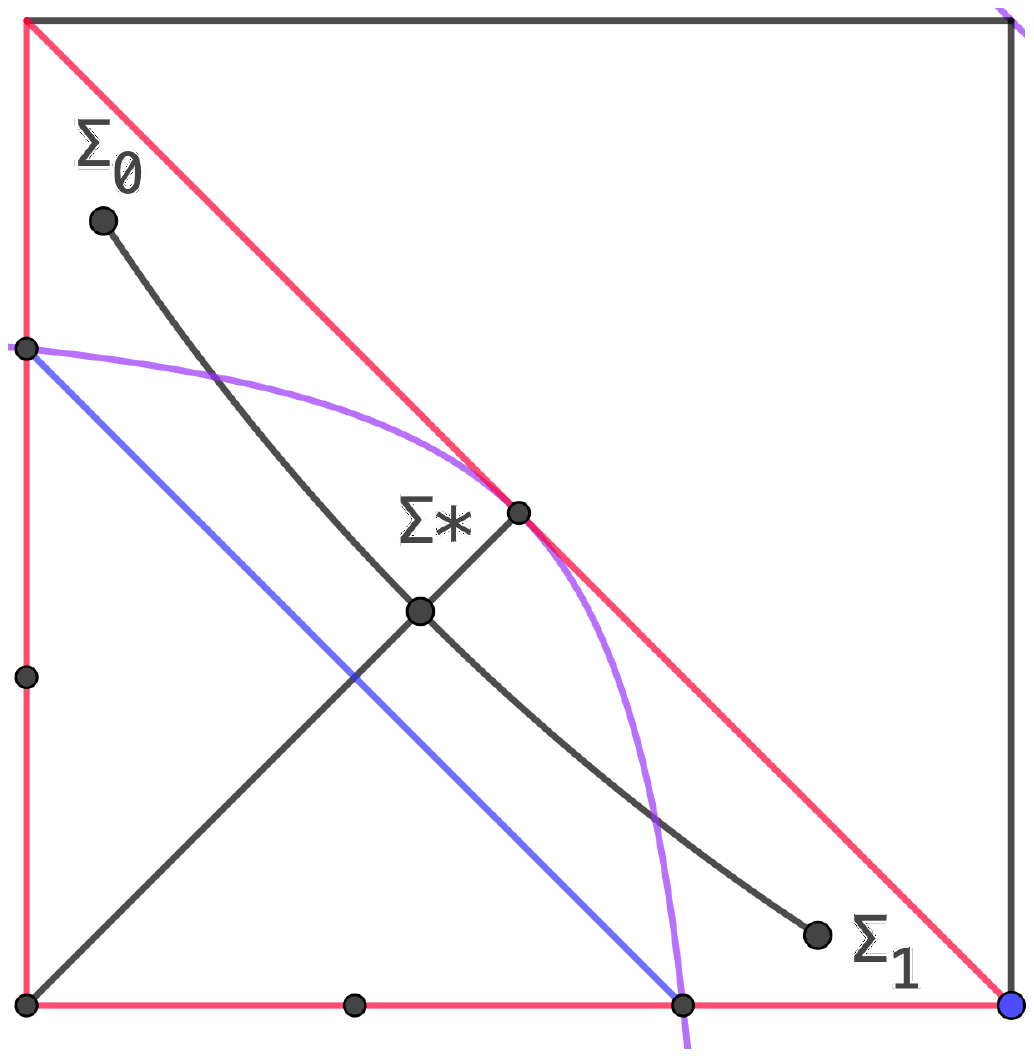} 
\vskip -1 in 
\caption{$ t \mapsto \left(\frac{ |\gaussp_L\left(\Sigma_t\right)|}{2\pi},\frac{ |\gaussp_R\left(\Sigma_t\right)|}{2\pi} \right)$  }
\end{center}
\end{figure}
\noindent Let us conclude the proof of Theorem \ref{thm4}.\\
Let $\Sigma$  be a minimal surface such that  the Gauss map area $|g(\Sigma)| <2\pi$. \\
Then  the projected area  verifies $|\gaussp_L(\Sigma)| +|\gaussp_R(\Sigma)| < 2\pi $ by  Corollary \ref{aires}.\\
Suppose $|\gaussp_L(\Sigma)|  = |\gaussp_R(\Sigma)| $ then by the  previous  inequality, $|\gaussp_L(\Sigma_{t_0} )| $ and $  |\gaussp_R(\Sigma_{t_0})| $ are each less than $\pi$ ie 
 the proportionate area of the  left  and right Gauss map  are less than $\frac{1}{2}$
and  by  Proposition  \ref{hypequi},  $\Sigma$ is stable.  \\
Suppose that  $|\gaussp_L(\Sigma)|  < |\gaussp_R(\Sigma)| $,  then  by equation \eqref{perm} : \\
$|\gaussp_L(\Sigma_1)|  =   |\gaussp_R(\Sigma)|   >  |\gaussp_L (\Sigma)| =   |\gaussp_R(\Sigma_1)| $.\\
By  continuity    there is  necessarily a  $t_0 \in [0,1[$ such that the left and right spherical area are equal:
$|\gaussp_L(\Sigma_{t_0} )|  = |\gaussp_R(\Sigma_{t_0})| $.  By Lemma \ref{propAsso}-3 and  Corollary \ref{aires}  each area is less than $\pi$  whence
$\Sigma_{t_0}$ is stable.  By Lemma \ref{propAsso}-4  $\Sigma$ is stable.
\vfill   \qed

\noindent{\fontsize{8}{8} \selectfont  DEPARTAMENTO DE MATEMATICA, UNIVERSITADE FEDERAL DE SANTA MARIA 97105-900 SANTA MARIA RS/BRAZIL}\\
{\it Email address: } {\tt  ari.aiolfi@ufsm.br}\\
\noindent{\fontsize{8}{8} \selectfont  INSTITUT DENIS POISSON , CNRS UMR 7013, UNIVERSITE DE TOURS, UNIVERSITE D'ORLEANS, PARC DE GRANDMONT 37200 TOURS}\\
{\it Email address: } {\tt  marc.soret@idpoisson.fr}\\
\noindent{\fontsize{8}{8} \selectfont UNIVERSITE PARIS EST CRETEIL, CNRS LAMA, F-94010 CRETEIL, FRANCE  }\\
{\it Email address: } {\tt  villemarina@yahoo.fr}\\

		\end{document}